\documentclass{amsart}
\usepackage{latexsym,amsmath,amscd, amssymb,latexsym,amsmath,amscd, amssymb, bickham}
\allowdisplaybreaks
\usepackage{calrsfs}
\DeclareMathAlphabet{\pazocal}{OMS}{zplm}{m}{n}

\newtheorem{theorem}{Theorem}[section]
\newtheorem{corollary}[theorem]{Corollary}
\newtheorem{lemma}[theorem]{Lemma}
\newtheorem{proposition}[theorem]{Proposition}
\theoremstyle{definition}
\newtheorem{definition}[theorem]{Definition}
\theoremstyle{notat}

\theoremstyle{claim}

\theoremstyle{remark}

\numberwithin{equation}{section}
\theoremstyle{example}
\newtheorem{example}[theorem]{Example}

\title{ c\MakeLowercase{yclotomic} \MakeLowercase{$p$-adic} m\MakeLowercase{ulti}-z\MakeLowercase{eta} v\MakeLowercase{alues} }
\author{S\MakeLowercase{\.{i}nan} \"{U}\MakeLowercase{nver}}
\address{Ko\c{c} University, Mathematics Department. Rumelifeneri Yolu, 34450, Istanbul, Turkey}
\address{Freie Universit\"{a}t Berlin, Mathematics Department, Arnimallee 3, 14195, Berlin, Germany}
\email{sunver@ku.edu.tr}

\begin{document}
\maketitle
\noindent

\begin{abstract}
The  cyclotomic $p$-adic multi-zeta values are the $p$-adic periods of  $\pi_{1}(\mathbb{G}_{m} \setminus \mu_{M},\cdot),$ the unipotent fundamental group of the multiplicative group minus the $M$-th roots of unity.    
In this paper,   we compute the cyclotomic $p$-adic multi-zeta values at all depths. This paper  generalizes  the results in  \cite{U2} and \cite{U3}. Since the main result gives quite explicit formulas we expect it to be useful in proving non-vanishing and transcendence results for these $p$-adic periods and also, through the use of $p$-adic Hodge theory, in proving non-triviality results for the corresponding $p$-adic Galois representations.  

\end{abstract}

\section{Introduction} 

There are not many examples of  motives over $\mathbb{Z}.$ The most basic examples of such motives are the Tate motives.  Another one is the unipotent completion of the fundamental group of the thrice punctured projective line $\pi_{1} (\mathbb{G}_{m} \setminus \{ 1\},\cdot),$ at a suitable tangential basepoint \cite{de}. In fact by a theorem of F. Brown, this motive generates the tannakian category of mixed Tate motives over $\mathbb{Z}.$ The complex periods of $\pi_{1} (\mathbb{G}_{m} \setminus \{ 1\},\cdot)$ are $\mathbb{Q}$-linear combinations of  the multi-zeta values given by 
$$
\zeta(s_{1},s_{2}, \cdots,s_{k}):=\sum _{0<n_{1}<\cdots <n_{k}}\frac{1}{n_{1} ^{s_{1}} n_{2} ^{s_{2}} \cdots n_{k} ^{s_{k}}},
$$
for $s_{1}, \cdots, s_{k-1} \geq 1$ and $s_{k}>1.$ These values were defined by Euler and studied by Deligne, Goncharov, Terasoma, Zagier etc.

Similarly, one can consider the unipotent fundamental group $\pi_{1}(\mathbb{G}_{m} \setminus \mu_{M}, \cdot)$ of the multiplicative group minus the group $\mu_{M}$ of $M$-th roots of unity for $M\geq 1.$ If $\pazocal{O}_{M}$ denotes the ring of integers of the $M$-th cyclotomic field, then this fundamental group defines a mixed Tate motive over $\pazocal{O}_{M}[1/M].$ The periods of this motive are linear combinations of  the cyclotomic multi-zeta values 
$$
\sum _{0<n_{1}<\cdots <n_{k}}\frac{\zeta^{i_1n_{1}+\cdots i_k n_k}}{n_{1} ^{s_{1}} n_{2} ^{s_{2}} \cdots n_{k} ^{s_{k}}},
$$
where $i_{j},$ for $1 \leq j \leq k,$ are fixed integers and $\zeta$ is an $M$-th root of unity. These values were studied and related to modular varieties and the theory of higher cyclotomy in \cite{gon}.

 This paper concerns the $p$-adic periods of the motive $\pi_{1}(\mathbb{G}_{m} \setminus \mu_{M}, \cdot).$  We have a realisation map from the category of mixed Tate motives over a number field to the category of mixed Tate filtered   $(\varphi,N)$-modules for any non-archimedean place of the number field \cite{cu}. Also for any (framed) mixed Tate filtered $(\varphi,N)$-module we  associate a period. The cyclotomic $p$-adic multi-zeta values (henceforth {\it cmv}'s) are the $p$-adic periods associated to the mixed Tate motive  defined by the unipotent fundamental group of $\mathbb{G}_{m} \setminus \mu_{M},$ for $p \nmid M.$ These values were defined in terms of the action of the crystalline frobenius on the fundamental group 
in \cite{U2},  generalising the notion of $p$-adic multi-zeta values (henceforth {\it pmv}'s) in \cite{U1}. In this paper we  give an explicit series  representation of  these $p$-adic periods. This is a generalisation of \cite{U3} to the cyclotomic case.

 We give an overview of the contents of the paper. In \textsection 2,  we start with studying certain types of series in terms of which the {\it cmv}'s will be expressed. These series can be of two types, denoted by  $\sigma$ or $\gamma,$ and are called the {\it cyclotomic p-adic iterated sum series} (or {\it ciss}). In fact the {\it ciss} are divergent and we will need to regularise them.  The regularisation  can be intuitively thought of as removing a combination of the summands which have large $p$ factors in the denominators that cause divergence. More precisely, we extend the algebra of $M$-power series functions by adding some highly divergent functions which we denote by $\sigma_{p}$ and we show in Proposition 2.9 that the {\it ciss} are contained in this algebra. 
 In Corollary 2.6, we show that the $\{ \sigma_{p}(\underline{s};\underline{i} )\}$'s form a basis for this extended algebra as a module over the algebra of $M$-power series functions. These two facts help us to define the regularised versions of the {\it ciss}, denoted by $\tilde{\sigma}$ and $\tilde{\gamma},$ in Definition 2.10. The limits of these regularised series are called the {\it cyclotomic p-adic iterated sums} (or {\it cis}), and denoted by $\underline{\sigma}$ and $\underline{\gamma}.$ 
Let $\zeta$ be a primitive $M$-th root of unity.  Let  $\pazocal{P}_{M}$ denote the    $\mathbb{Q}(\zeta)$-algebra generated by the {\it cis},  and $\pazocal{Z}_{M}$ the algebra generated by the {\it cmv}. The main theorem is 
\begin{theorem} 
We have the inclusion $\pazocal{Z}_{M} \subseteq \pazocal{P}_{M}.$
\end{theorem}
 The proof of this theorem occupies the whole of \textsection 3. The proof expresses in an inductive way every {\it cmv} as a series and should be thought of as an explicit computation of these values.

Finally, we would like to mention that Furusho defined in \cite{fur} another $p$-adic version of multi-zeta values that is essentially equivalent to ours in \cite{U1}. More precisely, the two versions generate the same algebra and each version can be obtained from the other one by elementary linear algebraic manipulations. This is explained in detail in \cite[Lemma 3.13]{U3}.  One can also define a version of cyclotomic version of Furusho's $p$-adic multi-zeta values which will again be essentially equivalent to the above version by the proof of \cite[Lemma 3.13]{U3}. 

{\it Acknowledgements.} This paper was written while the author was visiting H. Esnault at Freie Universit\"{a}t Berlin supported  by the fellowship for experienced researchers of the  Humboldt Foundation. The author thanks Prof. H. Esnault and the Humboldt Foundation for this support. 

\section{Cyclotomic p-adic iterated sum series}

Fix a prime $p$ and $M \geq 1,$ with $p \nmid M.$ Let $\zeta$ be a  primitive $M$-th root of unity, $K=\mathbb{Q}_{p}(\zeta)$ and $q,$ the cardinality of the residue field of $K.$       For $\underline{s}:=(s_{1},\cdots,s_{k}),$ with $0 \leq s_{i};$ $\underline{i}:=(i_{1},\cdots,i_{k})$ with $0\leq i_{j}<M;$ and $\underline{m}:=(m_{1},\cdots,m_{k}),$ with $0 \leq m_{i}<p,$ let  
$$
\sigma(\underline{s};\underline{i};\underline{m})(n):=\sum \frac{\zeta^{i_{1}n_{1}+\cdots i_{k}n_{k}}}{n_{1} ^{s_{1}} \cdots n_{k} ^{s_{k}}},
$$ where the sum is over $0<n_{1}<n_{2}<\cdots < n_{k}<n$ with $p|(n_{i}-m_{i}).$ If we let $\underline{n}:=(n_{1},\cdots, n_{k}) $ we will also write the numerator of the above summand as $\zeta^{\underline{i}\cdot \underline{n}}$ and the denominator as $\underline{n}^{\underline{s}}.$

Similarly, we let $\gamma(\underline{s};\underline{i};\underline{m})(n):=\frac{\zeta^{i_{k}n}}{n^{s_{k}}}\cdot\sigma(\underline{s}';\underline{i}';\underline{m}')(n),$ if $p|(n-m_{k})$ and 0 otherwise, with $\underline{s}'=(s_{1},\cdots,s_{k-1}),$ $\underline{i}':=(i_{1},\cdots, i_{k-1}),$  and $\underline{m}'=(m_{1},\cdots,m_{k-1}).$    Let $\sigma_{p}(\underline{s};\underline{i})(n):=\sigma(\underline{s};\underline{i};\underline{0})(n),$ where $\underline{0}=(0,\cdots,0).$ We define the {\it depth}  as  $d(\underline{s})=k$ and the {\it weight} as $w(\underline{s}):=\sum s_{i}.$

We call a sequence of the form $\sigma(\underline{s};\underline{i};\underline{m})$ or $\gamma(\underline{s};\underline{i};\underline{m})$ a {\it cyclotomic $p$-adic iterated sum series} (or {\it ciss}). 

\begin{definition}
Let $n \in \mathbb{N}$ and  let $f: \mathbb{N}_{\geq n} \to K$ be any function. We say that $f$ is an {\it $M$-power series function}, if there exist power series $p_{i} (x) \in K[[x]],$ which converge on $D(0,r_{i})$ for some $r_{i} > |p|$, for $0 < i  \leq pM,$ such that 
$
f(a)=p_{i}(a-i),
$ 
for all $a \geq n $ and $pM|(a-i).$ 
 \end{definition}
 Clearly there is a unique $M$-power series function with domain $\mathbb{N} $ and  which extends $f.$  We identify two $M$-power series functions if they agree on their common domains of definition. 
By the Weierstrass preparation theorem, the power series $p_{i}$ in the above definition are unique. Fix $0< l \leq pM,$ and let  $f$ be as above. Then there is a power series $p(x) \in K[[x]] $ which converges on some $D(0,r)$ with $r>|p|$ and 
$
f(lq^N)=p(lq^N),
$
for $N$ sufficiently large.

\begin{example}\label{exam main power}{\rm

 (i) Let $s \in \mathbb{Z}$ and $f(k):=\zeta^{ik} k^{s},$ for $p \nmid k$ and $f(k)=0$ for $p|k. $      Then $f$ is an  $M$-power series function.

(ii) Clearly the sums and products of  $M$-power series functions are  $M$-power series functions. 

(iii) Let $f$ be an $M$-power series function. For any $0 <l \leq pM,$ with $p|l$ let 
$$
f_{l}:= \lim _{n \to 0 \atop {pM | (n-l)}}f(n),
$$
with  $n$ ranging over positive integers such that $pM | (n-l),$ and tending to 0 in the $p$-adic metric.

Let $f^{[1]}$ be defined by 
$$
f^{[1]}(k)=\frac{f(k)-f_{l}}{k},
$$
if $p|k$ and  $pM|(k-l);$ and $f^{[1]}(k)=0,$ if $p \nmid k.$ 
We then see that  $f^{[1]}$ is an $M$-power series function. In fact, if $p|l,$ and $p$ is a power series around 0 such that $f(n)=p(n)$ for all $pM | (n-l)$ then $f^{[1]}(n)=q(n),$ for all $pM | (n-l),$ where   
$$
q(x)= \frac{p(x)-p(0)}{x}.
$$ 
Inductively, we let $f^{[k+1]}:=(f^{[k]})^{[1]}.$

(iv) Using the notation as above, let $f^{(1)}$ be defined by  $f^{(1)}(k):=f^{[1]}(k),$ if $p|k;$ and $f^{ (1)}(k)=\frac{f(k)}{k},$ if $p \not  |k.$ Then $f^{ (1)} $ is also an $M$-power series function. 

}
\end{example}

\begin{proposition}\label{prop main power}
Let $f: \mathbb{N}_{\geq n_{0}} \to K$ be an $M$-power series function. If we define $F:  \mathbb{N}_{\geq n_{0}} \to K$ by 
$$
F(n):= \sum _{n_{0}\leq  k\leq n}f(k)
$$
then $F$ is also an $M$-power series function. 
 \end{proposition}

The following lemma on power series will be essential while we are proving the linear independence of the $\sigma_{p}$'s.
 
\begin{lemma}
 Let $f,g \in K[[z]]$ be two power series which are  convergent on $D(a),$ for some $a>1.$ Suppose that $g \neq 0,$ and let $h:=f/g.$   If there exist  $a_{ij} \in K$ and $n\geq 1$ such that 
 $$
 h(z+M)-h(z)=\sum_{1\leq i\leq n \atop {0 \leq j < M }} \frac{a_{ij}}{(z+j)^{i}}
 $$ for infinitely many  $z \in D(a)$      then $h$ is constant and  $a_{ij}=0,$ for all $i$ and $j.$  
 \end{lemma}
 
\begin{proof}
The proof is a generalization of the proof of \cite[Lemma 2.0.2]{U3}. Note that by the Weierstrass preparation theorem the number of poles of $h$ on the closed unit disc $D(1)$ is finite.  This set is nonempty if at least one $a_{ij} \neq 0.$ Assume that this is the case and let this set be $\{ \alpha_{1},\cdots, \alpha_{k}\}.$ Arrange $\alpha_{i}$ so that $\alpha_{1}$ is a pole of $h(z),$ and hence $\alpha_{1} \in \{0, -1, \cdots, -(M-1) \}.$   Since $\alpha_{1}-M$ is not in the last set, it cannot be a pole of $h(z+M)-h(z),$ but since it is a pole of $h(z+M),$   it also has to be a pole of $h(z).$ Let $\alpha_{2}=\alpha_{1}-M.$ Continuing in this manner we will get $\alpha_{i}=\alpha_{1}-(i-1)M,$ and that $\alpha_{1}-kM$ is a pole of $h(z+M)-h(z)$ and hence is in $\{0, \cdots, -(M-1) \}.$ This is a contradiction. 
\end{proof}

Let $\mathcal{P}_{M}$ denote the algebra of $M$-power series functions which are 0 on $\mathbb{N} \setminus p\mathbb{N}.$ We will consider these as functions on $p\mathbb{N}.$   They are functions $f:p\mathbb{N} \to K$ such that there exist power series $p_{i},$ for $1\leq i \leq M,$ around 0 with radius of convergence greater than $|p|$ and which satisfy $f(pk)=p_{i}(pk)$ for $M|(k-i).$ Let us consider $\sigma_{p}(\underline{s};\underline{i})$ as functions on $p\mathbb{N}$ as well and let $\mathcal{P}_{M,\sigma}$ denote the module over $\mathcal{P}_{M}$ generated by the $\sigma_{p}(\underline{s};\underline{i})$   in $F(p\mathbb{N},K).$ This is an algebra as can be seen by using the shuffle product formula for series.  
  
  \begin{proposition}
The algebra $\mathcal{P}_{M,\sigma}$ is free with basis $\{\sigma _{p} (\underline{s},\underline{i}) | (\underline{s},\underline{i}) \in \cup_{n} (\mathbb{N}^{\times n}\times [0, M-1]^{\times n})  \}$ as a module over $\mathcal{P}_{M}.$ 
  \end{proposition}
 
 \begin{proof} 
 
 We will prove the linear independence of the set $S_{m}:=\{\sigma_{p}(\underline{s},\underline{i}) | d(\underline{s})\leq m  \} ,$ by induction on $m.$   For any function $f: p \mathbb{N} \to K,$ we let $\delta(f)$ denote the function defined by 
$\delta(f)(n):=f(n+p)-f(n).$    Note that 
 \begin{eqnarray}\label{eqndelta}
 \delta \sigma_{p}(\underline{s};\underline{i})(n)=\frac{\zeta^{i_{k}n}}{n^{s_{k}}} \sigma_{p}(\underline{s}';\underline{i}' )(n) .
 \end{eqnarray}
 
 Let $\delta_{M}(f)(n)=f(n+pM)-f(n).$ Then 
  \begin{eqnarray}\label{eqndeltaM}
 \delta_{M}(\sigma_{p}(\underline{s};\underline{i}))(n)=\sum_{0 \leq l <M}\frac{\zeta^{i_{k} (n+pl)}}{(n+pl)^{s_{k}}}  \sigma_{p}(\underline{s}';\underline{i}' )(n+pl).
\end{eqnarray}

 We know the linear independence for  the set $S_{0}=\{ 1\}.$ Assuming that  we know the linear independence for  $S_{m-1},$ we will prove it for  $S_{m}.$ Let us suppose that $\{\sigma_{p}(\underline{s};\underline{i})  \} \cup S_{m-1}$ is linearly dependent over $\mathcal{P}_{M}.$ Then there exists an $l'$ with $0\leq l' <M$ such that we have an expression of the form 
 $$
 \sigma_{p}(\underline{s};\underline{i})=\sum_{(\underline{t},\underline{j})  \atop {d(\underline{t}) \leq m-1} } a_{\underline{t},\underline{j}}\sigma_{p}(\underline{t};\underline{j}),
 $$
 which is valid for all $n$ which satisfies $pM|(n-pl')$ and 
 with $a_{\underline{t};\underline{j}}$ a quotient of power series  which converge on an open disc containing $|z|\leq |p|.$ 

 Applying $\delta_{M}$ to the last equation we get
 \begin{eqnarray*}
 & &\sum_{0 \leq l <M}\frac{\zeta^{i_{k} (n+pl)}}{(n+pl)^{s_{k}}}  \sigma_{p}(\underline{s}';\underline{i}' )(n+pl)=\\
 & &\Big(\sum_{(\underline{t},\underline{j})  \atop {d(\underline{t}) =m-1} } \delta_{M} (a_{\underline{t},\underline{j}}) \sigma_{p}(\underline{t};\underline{j})+ \sum _{(\underline{t},\underline{j})  \atop {d(\underline{t}) < m-1} }b_{\underline{t},\underline{j}}\sigma_{p}(\underline{t};\underline{j})\Big)(n),
 \end{eqnarray*}
 for $n$ which satisfies $pM|(n-pl').$  From the identity (\ref{eqndelta}) we see that $ \sigma_{p}(\underline{s}';\underline{i}' )(n+pl)$ is equal to $ \sigma_{p}(\underline{s}';\underline{i}' )(n)$ plus a linear combination of the terms $ \sigma_{p}(\underline{t};\underline{j}
  )(n),$ with $d(\underline{t})\leq m-2 $ and with coefficients which are quotients of power series. This together with the induction hypothesis implies that 
  $$
  \sum_{0 \leq  l<M} \frac{\zeta^{i_{k}(p(l'+l))}}{(n+pl) ^{s_{k}}}=\delta_{M}(a_{\underline{s}';\underline{i}'})(n),
  $$
 which contradicts the lemma above.

 Next we do an induction on the number of elements $\sigma_{p}(\underline{s},\underline{i})$ with $d(\underline{s})=m,$ and $a_{\underline{s},\underline{i}}\neq 0.$ Suppose that we have a non-trivial equation 
 $$
 \sum _{(\underline{s},\underline{i}) \atop {d(\underline{s}) \leq m} } a_{\underline{s},\underline{i}} \sigma_{p}(\underline{s},\underline{i})=0.
 $$
 By the induction assumption on $m,$ there is an $(\underline{s},\underline{i})$ with $d(\underline{s})=m$ such that $a_{\underline{s},\underline{i}}\neq 0.$ In particular, there exists an $0\leq l' <M$ such that $a_{\underline{s},\underline{i}}$ is not the zero function when restricted to $pl'+pM \mathbb{N}.$ In the remainder of the proof we will consider all the functions as functions on $pl'+pM \mathbb{N}.$  Dividing by $a_{\underline{s},\underline{i}}$ and rearranging we get 
 $$
 \sigma_{p}(\underline{s},\underline{i})+ \sum_{  (\underline{t},\underline{j}) \neq (\underline{s},\underline{i}) \atop {d(\underline{t})=m}} b_{\underline{t},\underline{j}}\sigma_{p}(\underline{t},\underline{j})= \sum _{(\underline{t},\underline{j})\atop {d(\underline{t})<m}} b_{\underline{t},\underline{j}}\sigma_{p}(\underline{t},\underline{j}),
 $$
 where $b_{\underline{t},\underline{j}}$ are quotients of power series. 
 Applying $\delta_{M}$ to this equation and using induction on the number of $b_{\underline{t},\underline{j}}\neq 0$ with $d(\underline{t})=m$ we obtain $\delta_{M}
 (b_{\underline{t},\underline{j}})=0$ for all $(\underline{t},\underline{j})$ with $d(\underline{t})=m, $ hence these $b_{\underline{t},\underline{j}}$ are constant and equal to, say $c_{\underline{t},\underline{j}}. $
 
 So the last equation can be rewritten as 
 $$
 \sigma_{p}(\underline{s},\underline{i})+ \sum_{(\underline{t},\underline{j}) \neq (\underline{s},\underline{i}) \atop {d(\underline{t})=m}} c_{\underline{t},\underline{j}}\sigma_{p}(\underline{t},\underline{j})= \sum _{(\underline{t},\underline{j})\atop {d(\underline{t})<m}} b_{\underline{t},\underline{j}}\sigma_{p}(\underline{t},\underline{j}).
 $$
 Applying $\delta_{M}$ and using the induction hypothesis to compare the coefficients of $\sigma_{p}(\underline{s}';\underline{i}')$ we obtain that 
 $$
p^{-s_{k}}\sum _{0 \leq l <M}\frac{\zeta^{i_{k}p(l+l')}}{(z+l)^{s_{k}}}+\sum_{(a,b)\neq (s_{k},i_{k})}c_{(\underline{s}',a;\underline{i}',b)}\sum_{0 \leq l <M}p^{-a}\frac{\zeta^{bp(l+l')}}{(z+l)^{a}}=\delta_{M}(b_{(\underline{s}';\underline{i}')}),
 $$
 where we put $pz=n.$ The previous lemma  then implies that the left hand side is equal to 0. Putting $\alpha_{b}:=c_{(\underline{s};\underline{i}',b)}$ and looking at the coefficient of $\frac{1}{(z+l)^{s_{k}}}$ we find that 
 $$
 \zeta^{i_{k}p(l+l')}+\sum_{b \neq i_{k}}\alpha_{b} \zeta^{bp(l+l')}=0,
 $$
for every $0\leq l <M.$ Rephrasing we see that there exist $\beta_{b}\in K,$ for $0\leq b <M$ with $\beta_{0}=1$ such that 
$$
\sum _{0\leq b<M}\beta_{b}\zeta^{lb}=0,
$$
for every $0\leq l <M.$   This contradicts the non-vanishing of the Vandermonde determinant for $\{1,\zeta, \cdots , \zeta^{M-1} \}$. 
\end{proof}

Let $\mathcal{F}_{M}$ denote the algebra of $M$-power series functions and $\iota \in \mathcal{F}_{M}$ denote the function that sends $n$ to $n.$ Let $\mathcal{F}_{M}(\frac{1}{\iota})$ be  the algebra obtained by inverting $\iota.$ Note that $\iota$ is already invertible on the components  $i+p\mathbb{N}$  with $0<i<p.$   Let $\mathcal{F}_{M,\sigma}$ be the module over $\mathcal{F}_{M}$ generated by  the $\sigma_{p}(\underline{s};\underline{i})$'s.  Then by the shuffle product formula for series, $\mathcal{F}_{M,\sigma }$ is an algebra. Let $\mathcal{F}_{M,\sigma} (\frac{1}{\iota})=\mathcal{F}_{M,\sigma} \otimes _{\mathcal{F}_{M}} \mathcal{F}_{M}(\frac{1}{\iota}).$

\begin{corollary}\label{freecorollary}
The algebra $\mathcal{F}_{M,\sigma}$ (resp. $\mathcal{F}_{M,\sigma}(\frac{1}{\iota})$) is free with basis $\{\sigma _{p} (\underline{s};\underline{i}) | (\underline{s},\underline{i}) \in \cup_{n} (\mathbb{N}^{\times n}\times [0, M-1]^{\times n})  \}$  as a module over $\mathcal{F}_{M}$ (resp. $\mathcal{F}_{M}(\frac{1}{\iota})$). 
\end{corollary}

\begin{proof} For a set $S,$ let $F(S,K)$ denote the algebra of functions from $S$ to $K.$ We have the following decomposition 
$$
F(\mathbb{N},K)=\oplus _{1\leq i \leq p} F(p\mathbb{N},K),
$$
where we send $f \in F(\mathbb{N},K)$ to the element on the right hand side whose $i$-th component is $f_{i} \in F(p\mathbb{N},K),$ defined by 
$$
f_{i}(k)=f(k-p+i),
$$
for $k \in p\mathbb{N}.$ We have $\sigma_{p}(\underline{s};\underline{i})_{i}=\sigma_{p}(\underline{s};\underline{i}),$ for all $1\leq i \leq p,$ where we abuse the notation and denote by $\sigma_{p}(\underline{s};\underline{i})$ both the function on the left hand side of the equality whose domain is $\mathbb{N}$ and also the function on the right hand side of the equation which is its restriction to $p\mathbb{N}.$ By the definition of the power series functions, the above decomposition gives the following decompositions:
$$
\mathcal{F}_{M}=\oplus _{1 \leq i \leq p} \mathcal{P}_{M}
$$
and 
$$
\mathcal{F}_{M,\sigma}=\oplus _{1 \leq i \leq p} \mathcal{P}_{M,\sigma}.
$$

Using this, the freeness of $\mathcal{F}_{M,\sigma}$ over $\mathcal{F}_{M}$ follows from Proposition 2.0.3 and the statement for $\mathcal{F}_{M,\sigma}(\frac{1}{\iota})$ follows by localization. 
\end{proof}

\begin{definition}
Let $\mathfrak{r}: \mathcal{F}_{M,\sigma}\to \mathcal{F}_{M}$ denote the projection with respect to the above basis. We will denote the projection $\mathcal{F}_{M,\sigma}(\frac{1}{\iota})\to \mathcal{F}_{M}(\frac{1}{\iota})$ by the same notation. Similarly, let $\mathfrak{s}: \mathcal{F}_{M}(\frac{1}{\iota}) \to \mathcal{F}_{M}$ denote the projection that has the effect of deleting the principal parts of the  Laurent series expansions around 0 for the components $p\mathbb{N},$ and is identity on the components $i+p\mathbb{N}$ with $0<i<p.$ 

\end{definition}

Let $\underline{s}:=(s_{1},\cdots,s_{k}),$ and $\underline{t}:=(t_{1},\cdots,t_{l}).$ We write $\underline{t} \leq \underline{s}$ if there  exists an increasing function $j: \{1,\cdots, l \} \to \{1,\cdots , k \}$  such that $t_{i} \leq s_{j(i)},$ for all $i.$

\begin{lemma}
 Let $f$ be an $M$-power series function and let $g$ be defined as 
 $$
 g(n)=\sum _{0<a<n} f(a)\sigma_{p}(\underline{s};\underline{i})(a)
 $$
 for some $\underline{s}:=(s_{1},\cdots,s_{k})$ and $\underline{i}:=(i_{1},\cdots, i_{k}).$ Then 
 $$
 g=\sum _{(\underline{t},\underline{j})\atop {\underline{t} \leq \underline{s}}} f_{\underline{t},\underline{j}} \sigma_{p}(\underline{t},\underline{j}),
 $$
 for some $M$-power series functions $f_{\underline{t},\underline{j}}.$ Similarly, 
 if $h$ is defined as 
 $$
 h(n):=\sum_{0<a<n \atop {p|a}} \frac{f(a)}{a^{s}} \sigma_{p}(\underline{s};\underline{i})(a),
 $$ for some $s\geq 1$ then 
 $$
 h=\sum _{ (\underline{t},\underline{j}) \atop {\underline{t} \leq \underline{s}'}} f_{\underline{t},\underline{j}} \sigma_{p}(\underline{t};\underline{j}),
 $$
 for some $M$-power series functions $f_{\underline{t},\underline{j}},$ 
 where $\underline{s}':=(s_{1},\cdots,s_{k},s).$ 
 \end{lemma}

\begin{proof} 
We will prove this by induction on $d(\underline{s}).$ Suppose that $d(\underline{s})=0$ and hence $\sigma_{p}(\underline{s};\underline{i})=1.$ Then for $g$ the assertion follows from Proposition \ref{prop main power}. For $0 \leq l <M,$ let  $f_{l}$ be the power series in $K[[z]]$ which has the property that $f(n)=f_{l}(n)$ for $n$ such that  $p|n$ and $M|(n-l).$ Write  $f_{l}(z)=\sum _{0 \leq i}b_{il} z^{i},$ for $|z|\leq |p|$  then 
$$
h(n)=\sum _{0 \leq l <M} \sum _{0 \leq i <s} \sum_{0 <a<n \atop p|a, \, M|(a-l)}\frac{b_{il}}{a^{s-i}}+\sum _{0<a<n }\overline{f}(a),
$$
 where $\overline{f}$ is the unique $M$-power series function which satisfies $
\overline{f}(n)=0$ if $p \nmid n$ and $\overline{f}(n)=\sum _{s\leq i}b_{il}n^{i-s}$ if $p|n$ and $M|(n-l).$ Then Proposition \ref{prop main power} implies that the second sum defines an $M$-power series function. In order to see that $h$ is an $M$-power series functions it suffices to show that the function
$$t(n):=\sum _{0<a<n \atop {p|a, \, M|(a-l)} }\frac{1}{a^{t}},$$
for any $0\leq l <M,$ is a $K$-linear combination of the $\sigma_{p}(t;i)$'s for $0\leq i <M.$ This follows immediately from the fact that the characters $\chi_{i}:\mathbb{Z}/M \to K$ defined by $\chi_{i}(\alpha)=\zeta^{i\alpha}$ are distinct for $0\leq i <M$ and hence are $K$-linearly independent.

Now assume the statement for all $(\underline{s},\underline{i})$ with $d(\underline{s})\leq k$ and fix $\underline{s}:=(s_{1},\cdots,s_{k+1})$ and $\underline{i}:=(i_{1},\cdots,i_{k+1}).$  Let $F$ be as in Proposition \ref{prop main power},  then 
$$
g(n)=F(n-1)\sigma_{p}(\underline{s};\underline{i})(n)-\sum_{0<n_{k+1}<n \atop {p|n_{k+1}}}F(n_{k+1})\frac{\zeta^{i_{k+1}n_{k+1}}\sigma_{p}(\underline{s}';\underline{i}')(n_{k+1})}{n_{k+1}^{s_{k+1}}}
$$ 
 and the statement follows from the induction hypothesis on $h.$
 
On the other hand, to prove the statement on $h,$ we write $h(n)=$
\begin{eqnarray*}
\sum _{0 \leq l <M} \sum _{0 \leq i <s} \sum_{0 <a<n \atop p|a, \, M|(a-l)}\frac{b_{il}}{a^{s-i}}\sigma_{p}(\underline{s};\underline{i})(a)+\sum_{0<a<n }\overline{f}(a)\sigma_{p}(\underline{s};\underline{i})(a),
\end{eqnarray*}
using the notation above. The second summand defines a function which is of the form as in the statement of the lemma because of the induction hypothesis on $g.$ To finish the proof, it suffices  to show that  the function which sends $n$ to 
$$
\sum _{0<a<n \atop {p|a, \, M|(a-l)}}\frac{1}{a^t}\sigma_{p}(\underline{s};\underline{i})(a)
$$
is a $K$-linear combination of the functions $\sigma_{p}(\underline{s},t;\underline{i},j),$ for $0\leq j <M.$ We prove this exactly as above. 
 \end{proof}

\begin{proposition}
For any $\underline{s}$ and $\underline{m},$ $\sigma(\underline{s};\underline{i};\underline{m}) \in \mathcal{F}_{M,\sigma}.
$
\end{proposition}

\begin{proof}
We will prove this by induction on $d(\underline{s}).$ If $d(\underline{s})=1,$ then $\sigma(\underline{s};\underline{i};\underline{m})=\sigma_{p}(\underline{s};\underline{i})$ if $m_{1}=0;$ and  $\sigma(\underline{s};\underline{i};\underline{m})\in \mathcal{F}_{M}$ otherwise, by Proposition \ref{prop main power}. Suppose we know the result for $d(\underline{s})\leq k,$ and fix $\underline{s}$ with $d(\underline{s})=k+1.$  

 Since 
$$\sigma(\underline{s};\underline{i};\underline{m})(n)=\sum _{0<a<n \atop {p| (a-m_{k+1})}} \frac{\zeta^{ai_{k+1}}\sigma(\underline{s}';\underline{m}')(a)}{a^{s_{k+1}}},$$
using the induction hypothesis we realize that we only need to show that functions of the   form 
$$
\sum _{0<a<n \atop {p| (a-m)}}\frac{f(a)}{a^{s}}\sigma_{p}(\underline{t};\underline{j})(a),
$$
with $f$ an  $M$-power series function, are in $\mathcal{F}_{M,\sigma}$ and this is exactly the statement of the previous lemma. 
 \end{proof}

In fact, from the proof above it follows that $\sigma(\underline{s};\underline{i};\underline{m})$  is an $\mathcal{F}_{M}$-linear combination of $\sigma_{p} (\underline{t};\underline{j})$ with $\underline{t} \leq \underline{s}.$  
 
 \begin{definition}
 For a function $f \in \mathcal{F}_{M,\sigma},$ let $\tilde{f}:=\mathfrak{r}(f) \in \mathcal{F}_{M}.$  We call $\tilde{f}$ the {\it regularization}  of $f.$ Since by the previous proposition $\sigma(\underline{s};\underline{i};\underline{m}) \in \mathcal{F}_{M,\sigma},$ we let $\tilde{\sigma}(\underline{s};\underline{i};\underline{m})\in \mathcal{F}_{M}$ be its regularization and for $0< l \leq M,$ we let 
 $\underline{\sigma}(\underline{s};\underline{i};\underline{m})[l]:=\lim_{N \to \infty}\tilde{\sigma}(\underline{s};\underline{i};\underline{m})(lq^N)$ and  $\underline{\sigma}(\underline{s};\underline{i};\underline{m}):=\underline{\sigma}(\underline{s};\underline{i};\underline{m})[1].$

 \end{definition}
 
 For a function $f: \mathbb{N} \to K$ and $0 \leq m <p,$ let $f_{[m]}$ denote the function which is equal to $f$ for values $n$ which are congruent to $m$ modulo $p$ and is 0 otherwise.  Recall that $\gamma(\underline{s};\underline{i};\underline{m})(n):=\zeta^{ni_{k}}n^{-s_{k}}\cdot\sigma(\underline{s}';\underline{i}';\underline{m}')_{[m_{k}]}(n).$  We will define the regularized version $\tilde{\gamma}(\underline{s};\underline{i};\underline{m})$ of $\gamma(\underline{s};\underline{i};\underline{m})$ as follows. If $m_{k} \neq 0,$ then it is defined as $\tilde{\gamma}(\underline{s};\underline{i};\underline{m})(n)=\zeta^{ni_{k}}n^{-s_{k}}\cdot\tilde{\sigma}(\underline{s}';\underline{i}';\underline{m}')_{[m_{k}]}(n).$ If $m_{k}=0,$ and for $0 \leq l <M,$ $p_{l}(z)=a_{0l}+a_{1l}z+\cdots$ is such that $\tilde{\sigma}(\underline{s}';\underline{i}';\underline{m}')(n)=p_{l}(n)$ for $p|n$ and $M|(n-l),$ then 
 $\tilde{\gamma}(\underline{s};\underline{i};\underline{m})(n):=\zeta^{ni_{k}}(a_{s_{k}l}+a_{s_{k}+1,l }n+\cdots) ,$ if $p|n$ and $M|(n-l)$ and 0 if $p \nmid n.$ Finally, we let $\underline{\gamma}(\underline{t};\underline{i};\underline{m})[l]=\lim_{N \to \infty}\tilde{\gamma}(\underline{t};\underline{i};\underline{m})(lq^{N})=\zeta^{li_{k}}a_{s_{k}l}$ and $\underline{\gamma}(\underline{t};\underline{i};\underline{m}):=\underline{\gamma}(\underline{t};\underline{i};\underline{m})[1].$ 
 
 Another way to describe this is as follows. For any $\underline{s}, \underline{i}$ and $\underline{m}, $ $\gamma(\underline{s};\underline{i};\underline{m}) \in \mathcal{F}_{M,\sigma}(\frac{1}{\iota}),  $  and $\tilde{\gamma}(\underline{s};\underline{i};\underline{m}):=\mathfrak{s} \circ \mathfrak{r}(\gamma(\underline{s};\underline{i};\underline{m})).$
  

\begin{definition}
Let $\pazocal{P}_{M}$ (resp. $\pazocal{S}_{M},$  $\tilde{\pazocal{S}}_{M}$) denote the   $\mathbb{Q}(\zeta)$-algebra (resp. vector space) spanned  by the $\underline{\sigma}(\underline{s};\underline{i};\underline{m})$ (resp. $\sigma(\underline{s};\underline{i};\underline{m}),$  $\tilde{\sigma}(\underline{s};\underline{i};\underline{m})$) and the $\underline{\gamma}(\underline{s};\underline{i};\underline{m})$  (resp. $\gamma(\underline{s};\underline{i};\underline{m}),$  $\tilde{\gamma}(\underline{s};\underline{i};\underline{m})$). \end{definition}

We call  $p$-adic numbers of the form $\underline{\sigma}(\underline{s};\underline{i};\underline{m})$ or $\underline{\gamma}(\underline{s};\underline{i};\underline{m}),$ the {\it cyclotomic p-adic iterated sums} (or {\it cis}).

\section{proof of theorem 1.1}

\subsection{Cyclotomic $p$-adic multi-zeta values} We  recall   notation and concepts from \cite{U2}.
Fix $M \geq 1,$ and $p \nmid M.$ Let $K \langle \langle e_{0}, \cdots, e_{M} \rangle \rangle$ denote the ring of non-commutative power series in the variables $e_{0}, e_{1}, \cdots, e_{M}.$  Studying the action of the crystalline frobenius on the fundamental group of $\mathbb{G}_{m} \setminus \mu_{M},$ we defined, for every $1 \leq i \leq M,$ $g_{i} \in K \langle \langle e_{0}, \cdots, e_{M} \rangle \rangle$ \cite[\textsection 2.2.3]{U2}. For an element $\alpha \in K \langle \langle e_{0}, \cdots, e_{M} \rangle \rangle$ and any monomial $e^I=e_{i_{1}}\cdots e_{i_n},$ let $\alpha[e^I]$ denote the coefficient of $e^I$ in $\alpha.$ If $e^I=e_{i_{1}}\cdots e_{i_n},$ we call $w(e^I)=w(e_{i_{1}}\cdots e_{i_n}):=n,$ the {\it weight} of $e^I.$ By \cite[(2.2.7)]{U2}, we see that $\{g_{i}[e^I] | I \} =\{g_{j}[e^{I}] | I \},$ for any $i,j.$ Therefore it makes sense to study only one of the $g_{i}$'s. We let $g:=g_{M},$ and we defined the {\it cyclotomic p-adic multi-zeta values} (or {\it cmv}) as the coefficients $g[e_{i_1} \cdots e_{i_n}],$ and we used the notation 
$$
g[e_{0} ^{s_k -1}e_{i_k}\cdots e_{0} ^{s_1 -1}e_{i_1}  ]=p^{\sum s_i} \zeta_{p}(s_{k}, \cdots, s_1;i_k,\cdots , i_1),
$$
where $1 \leq i_{1}, \cdots , i_k \leq M.$ We call $k$ the {\it depth} of the monomial  $e_{0} ^{s_k -1}e_{i_k}\cdots e_{0} ^{s_1 -1}e_{i_1} $ or the corresponding {\it cmv}, and denote it by $d(e^I).$ 
 
Let $\pazocal{U}_{M}$ denote the affinoid that is obtained by removing  discs of radius one in $\mathbb{P}^{1}_{K}$ around every $M$-th root of unity. Let $\pazocal{A}_{M}$ denote the algebra of rigid analytic functions on $\pazocal{U}_{M}.$ Then choosing the lifting $\pazocal{F}$ of frobenius given by $\pazocal{F}(z)=z^p, $ defines  a corresponding  element $\mathbcal{g}_{\pazocal{F}} \in \pazocal{A}_{M} \langle \langle e_{0}, \cdots, e_{M}  \rangle \rangle.$ Let $\omega_{0}:=d log (z)$ and $\omega_{i}:=dlog(z-\zeta ^{i}),$ for $1 \leq i \leq M.$ For $1 \leq i \leq M,$ let $\underline{i}$ be the unique integer such that $M| (i-p \underline{i}).$ Then in \cite[(2.2.10)]{U2},  we proved the following fundamental differential equation for $\mathbcal{g}_{\pazocal{F}}:$
 \begin{eqnarray*}\label{funddiffeq1}
d \mathbcal{g}_{\pazocal{F}}=\sum e_{i} \pazocal{F}^{*}\omega_{i} \cdot \mathbcal{g}_{\pazocal{F}}-\mathbcal{g}_{\pazocal{F}}\cdot \sum p g_{i} ^{-1}e_{i}g_{i} \omega_{\underline{i}},
\end{eqnarray*}
where  the sums  are over $0 \leq i \leq M$ and $g_{0}:=1.$ 
We can rewrite this as follows, 
\begin{eqnarray}\label{funddiffeq2}
d \mathbcal{g}_{\pazocal{F}}[e^I]=\pazocal{F}^*\omega_{a}\mathbcal{g}_{\pazocal{F}}[e^{I'}]-p\sum_{i,J,K}(\mathbcal{g}_{\pazocal{F}}g_{i} ^{-1})[e^J]g_{i}[e^K]\omega_{\underline{i}}
\end{eqnarray}
where $I=(a,I'),$ and  the second sum runs over $J,K$ and $0\leq i \leq M$ such that $(J,i,K)=I.$

Let us $h$ denote $\mathbcal{g}_{\pazocal{F}}(\infty).$ Then we proved the following equation in \cite[(4.1.1)]{U2}  that relates $h$ and the $g_{i}$'s: 
\begin{eqnarray}\label{relating h and g}
h\cdot \sum  g_{i} ^{-1}e_{i}g_{i}=\sum  e_{i} \cdot h,
\end{eqnarray}
where the sums are over $0 \leq i \leq M.$ 

For   $\alpha \in K[[ z]]\langle \langle e_{0}, \cdots, e_{n} \rangle \rangle,$ and a monomial  $e^I,$ note that $\alpha[e^I] \in K[[ z]]$ is the coefficient of $e^I$ in $\alpha.$ We let $\alpha \{ e^I\}$ denote the function from $\mathbb{N}$ to $K$ that sends $n$ to the coefficient of $z^n$ in $\alpha[e^I]. $ If $\alpha \in \pazocal{A}_{M}\langle \langle e_{0}, \cdots, e_{n} \rangle \rangle, $ we define $\alpha\{ e^I\} $ by first viewing $\alpha$ in $K[[ z]]\langle \langle e_{0}, \cdots, e_{n} \rangle \rangle,$ by expanding around the origin.

\subsection{Proof of Theorem 1.1} In order to prove Theorem 1.1, we need to show that $g_i[e^I] \in \pazocal{P}_{M},$ for every monomial $e^I$ and $1 \leq i \leq M.$   We will prove this together with the statement that $\mathbcal{g}_{\pazocal{F}}\{e^I \} \in \pazocal{P}_{M}\cdot \tilde{\pazocal{S}}_{M}.$ The proof will be by induction on the weight of $e^{I}.$  We will first show that $\mathbcal{g}_{\pazocal{F}}\{ e^I \} \in K\cdot \pazocal{S}_{M},$ then we will prove in fact that it lies in $K \cdot \tilde{\pazocal{S}}_{M}$ and finally in $\pazocal{P}_{M} \cdot \tilde{\pazocal{S}}_{M}.$ 

We will prove the following statements together by induction on $w:$

(i) $\mathbcal{g}_{\pazocal{F}}\{ e^I \} \in \pazocal{P}_{M}\cdot \tilde{\pazocal{S}}_{M},$ for $w(I) \leq w$

(ii) $h[e^J] \in \pazocal{P}_{M}$ if $w(J) \leq w-1.$ 

and 

(iii) $g_{i}[e^J] \in \pazocal{P}_{M}$ if $w(J) \leq w-1$

$\bullet$ Let us look at the statements (i), (ii) and (iii) for $w=1.$ 

From $d \mathbcal{g}_{\pazocal{F}}[e_{0}]=0,$ we see that $\mathbcal{g}_{\pazocal{F}}[e_{0}]=0.$ Similarly, from $d\mathbcal{g}_{\pazocal{F}}[e_{a}]=\pazocal{F}^{*}\omega_{a}-p\omega_{\underline{a}},$ we see that 
$$
\mathbcal{g}_{\pazocal{F}}[e_{a}](z)=p \sum _{0 <n \atop {p \nmid n}}\frac{(\zeta^{-\underline{a}}z)^n}{n}.
$$

From this we see that $(i)$ is valid for $w=1;$ as for $(ii)$ and $(iii),$ they are trivially true for $w=1.$

$\bullet$ Assume that we know (i), (ii) and (iii) for $w.$ We will prove them for $w$ replaced with $w+1.$

 Note that  by the induction assumption $\mathbcal{g}_{\pazocal{F}}\{ e^J \} \in \pazocal{P}_{M} \cdot \tilde{\pazocal{S}}_{M}\subseteq K\cdot \pazocal{S}_{M},$ for $w(J) \leq w.$ This implies that $\mathbcal{g}_{\pazocal{F}}\{e^I \} \in K\cdot \pazocal{S}_{M},$ if $w(I)=w+1,$ by the differential equation  (\ref{funddiffeq2}). 
 
 By construction \cite[\textsection 2.2.4]{U2}, $\mathbcal{g}_{\pazocal{F}}[e^I]$ is a rigid analytic function on $\pazocal{U}_{M}.$ Therefore by \cite[Corollary 3.0.4]{U2}, for any $0 \leq l <pM,$ if $\lim_{N \to \infty}lq^{N}\mathbcal{g}_{\pazocal{F}}\{e^I\}(lq^N)$ exists then it is equal to 0. 

Now note that by the induction assumption $\mathbcal{g}_{\pazocal{F}}\{e^J\} \in \pazocal{P}_{M} \cdot \tilde{\pazocal{S}}_{M}\subseteq K\cdot \tilde{\pazocal{S}}_{M},$ for $w(J) \leq w.$ In particular,  $\mathbcal{g}_{\pazocal{F}}\{e^J \} $ is an $M$-power series function. Then the differential equation shows that the function which sends $n$ to $n\cdot \mathbcal{g}_{\pazocal{F}}\{e^I\}(n)$ defines an $M$-power series function by Proposition \ref{prop main power}. This implies that the limits $\lim_{N \to \infty}lq^{N}\mathbcal{g}_{\pazocal{F}}\{e^I\} (lq^N)$ exist, for any $0 \leq l <pM,$ and therefore they are 0. This together with the above fact that  $n\cdot \mathbcal{g}_{\pazocal{F}}\{e^I\} ( n)$ is an $M$-power series function then implies that $\mathbcal{g}_{\pazocal{F}}\{e^I \}( n)$ is an $M$-power series function. Therefore, we have $\mathbcal{g}_{\pazocal{F}}\{ e^I\} \in K\cdot \tilde{\pazocal{S}}_{M}.$

 Now reinterpreting the fact that $\lim_{N \to \infty}q^N \mathbcal{g}_{\pazocal{F}}\{e^I\} ( q^N)=0,$ using the differential equation (\ref{funddiffeq1}) for $d \mathbcal{g}_{\pazocal{F}}[e^I] ,$ we see, by the induction hypotheses and the definition of $\pazocal{P}_{M},$  that with $e^I=e_{a}e^Je_b:$  

(a) if $1 \leq a,b\leq M,$ then we get 
$$
\zeta^{-\underline{a}}g_{a}[e^Je_{b}]-\zeta^{-\underline{b}}g_{b}[e_{a}e^{J}]\in \pazocal{P}_{M}
$$

(b) If $1 \leq a \leq M$ and $b=0$ then 
$$g_{a}[e^{J}e_{0}] \in \pazocal{P}_{M}$$ 

(c) If $1 \leq b \leq M$ and $a=0$ then 
$$
 g_{b}[e_{0}e^{J}] \in \pazocal{P}_{M}.
 $$

(d) If $a=b=0,$ we do not get any new information.

Using (a)-(c) we immediately see the following lemma.

\begin{lemma}\label{lemmae0} If $1 \leq i \leq M,$ and $R$ is of weight $w,$ and such that $e^R$ contains an $e_{0}$ factor then $g_{i}[e^R]\in \pazocal{P}_{M}.$ 
\end{lemma}
This lemma together with the relation (\ref{relating h and g}) implies the statement (ii) above for $w$ replaced with $w+1:$ 

\begin{proposition}\label{lemmah}
If $R$ has weight $w,$ then 
$h[e^R] \in \pazocal{P}_{M}.$ 
\end{proposition}

\begin{proof}

Now for any $e^R$  with $w(R)=w(I)-1$ let us look at the coefficients of $e_0e^R$ on both sides of the identity 
 $$
h\cdot \sum _{0 \leq i \leq M} g_{i} ^{-1}e_{i}g_{i}=\sum _{0\leq i \leq M} e_{i} \cdot h
$$  
to get 
$$
h[e^R]-(hg_{r} ^{-1})[e_0e^{R'}]\in \pazocal{P}_{M}
$$
by the induction hypotheses on $h$ and $g_{a},$ where $e^R=e^{R'}e_{r}.$ Again by this hypothesis we see that 
$$
(hg_{r} ^{-1})[e_0e^{R'}]- (h[e_0e^{R'}]-g_r[e_0e^{R'}] )\in \pazocal{P}_{M}.
$$
Noting that $g_{r}[e_{0}e^{R'}] \in \pazocal{P}_{M}$ we arrive at 
$$
h[e^R]- h[e_0e^{R'}] 
\in \pazocal{P}_{M}.
$$
Replacing $e^R$ with $e_0e^{R'}$ above we see that 
$$
h[e_0e^{R'}] -h[e_0 ^2e^{R''}] \in \pazocal{P}_{M}
$$
where $e^{R'}=e^{R''}e_{r'}.$ Proceeding in this manner and adding all the terms we obtain 
$$
h[e^R]-h[e_{0} ^w] \in \pazocal{P}_{M},
$$
where $w$ is the weight of $e^R.$ Since $h[e_{0} ^w]=\frac{h[e_0] ^w}{w!}=0,$ we have 
$$
h[e^R] \in\pazocal{P}_{M}.
$$
\end{proof}

Let us  continue with the proof of (iii) for $w$ replaced with $w+1.$ We need to show that $g_{i}[e^J] \in \pazocal{P}_{M}$ for $w(J)=w.$ By the above we know this statement if $e^J$ has an $e_{0}$ factor. 

Suppose that $R$ has weight $w-1$ and let us look at the coefficients of $e_a e^Re_b$  in the identity
$$
h\cdot \sum _{0 \leq i \leq M} g_{i} ^{-1}e_{i}g_{i}=\sum _{0\leq i \leq M} e_{i} \cdot h
$$  
to obtain
$$
(hg_{b} ^{-1})[e_ae^R]+g_{a}[e^Re_b] -h[e^Re_b] \in \pazocal{P}_{M},
$$
by Proposition \ref{lemmah} and the induction assumption on $g_{i}.$ Simplifying further using the same results we have
\begin{eqnarray}\label{keyidentity}
g_{a}[e^Re_b]-g_{b} [e_ae^R] \in \pazocal{P}_{M}.
\end{eqnarray}

Now we can prove (iii) for  $w$ replaced with $w+1:$

\begin{proposition}
If $w(J)=w$ then $g_{i}[e^J] \in \pazocal{P}_{M}$ for any $1\leq i \leq M.$ 
\end{proposition}

 \begin{proof} 
 We proved  the statement if $e^J$ has an $e_{0}$ factor. 
 Note that so far we have seen that if $R$ has weight $w-1$ then for any $a$ and $b$ 
 $$
 \zeta^{-\underline{a}}g_{a}[e^Re_{b}]-\zeta^{-\underline{b}}g_{b}[e_{a}e^R] \in \pazocal{P}_{M}
 $$ 
 and
 $$
 g_{a}[e^Re_b]-g_{b} [e_ae^R] \in \pazocal{P}_{M}.
 $$
 This proves the statement in case $e^J$ does not begin or end with $e_{i}.$ The case when $e^{J}=e_{i} ^{w}$ is trivially true since $g_{i}[e_{i} ^{w}]=0.$ In the remaining case  we can  write $e^J=e_{i} ^{r}e^Se_{c} e_{i} ^{s}$ for some nonzero $c \neq i$ and $r,s \geq 1.$ 
 Applying (\ref{keyidentity}) $s$-times and adding the terms we see that 
 $$
 g_{i}[e^J]-g_{i}[e_{i} ^{r+s}e^Se_{c}] \in \pazocal{P}_{M}.
 $$
 Since $c \neq i,$ by the above discussion we know that $g_{i}[e_{i} ^{r+s}e^Se_{c}] \in \pazocal{P}_{M}.$ This finishes the proof that $g_{i}[e^J] \in \pazocal{P}_{M}.$ 
 \end{proof}

Finally, we prove the statement (i) for $w$ replaced with $w+1.$ Let $e^I$ be a monomial  of weight $w+1.$ We have seen above that $\mathbcal{g}_{\pazocal{F}}\{e^{I}\} \in K \cdot \tilde {\pazocal{S}}_{M}.$ We also know by the induction assumption that $\mathbcal{g}_{\pazocal{F}}\{e^{J}\} \in \pazocal{P}_{M} \cdot \tilde {\pazocal{S}}_{M},$ for any $J$ of weight less than or equal to $w.$  
This, together with the fact we just proved that $g_{i}[e^{J}] \in \pazocal{P}_M,$ for any $1 \leq i \leq M$ and $J$ of weight less than or equal to $w,$ implies that all the coefficients that appear in the differential equation for $d \mathbcal{g}_{\pazocal{F}}[e^{I}]$ lie in $\pazocal{P}_{M}.$ This implies that $\mathbcal{g}_{\pazocal{F}}\{e^{I}\} \in \pazocal{P}_{M} \cdot \tilde {\pazocal{S}}_{M},$ proving the claim and  finishing the proof of Theorem 1.1.

\end{document}